\documentclass[10pt]{amsart}
\usepackage{amscd,amsmath,amssymb,amsthm,enumerate,graphicx}
\newtheorem{definition}[equation]{Definition}
\newtheorem{lemma}[equation]{Lemma}
\newtheorem{proposition}[equation]{Proposition}
\newtheorem{theorem}[equation]{Theorem}
\newtheorem{corollary}[equation]{Corollary}

\newcommand\lemmaref[1]{Lemma~\ref{#1}}

\newcommand\theoremref[1]{Theorem~\ref{#1}}
\newcommand\corollaryref[1]{Corollary~\ref{#1}}

\title{Classification of Klein Four Symmetric Pairs of Holomorphic Type for $\mathrm{E}_{7(-25)}$}
\author{Haian HE}
\date{}
\address{Department of Mathematics,
College of Sciences, Shanghai University,
No.\ 99 Shangda Road, Baoshan District,
Shanghai, China P.\ R.\ 200444}
\email{hebe.hsinchu@yahoo.com.tw}

\begin{document}
\begin{abstract}
The author classifies Klein four symmetric pairs of holomorphic type for the non-compact Lie group of Hermitian type $\mathrm{E}_{7(-25)}$, and applies the results to branching rules.
\end{abstract}
\maketitle
\section{Introduction}
\subsection{Notations}
Denote by $\mathbb{R}$ and $\mathbb{C}$ the real field and the complex field respectively. For any positive integer $n$, denote by $I_n$ the $n\times n$ identity matrix. For a Lie group $G$ with its Lie algebra $\mathfrak{g}$, denote by $\mathrm{Aut}G$ and $\mathrm{Aut}\mathfrak{g}$ the automorphism group of $G$ and $\mathfrak{g}$ respectively. Also, denote by $\mathrm{Int}\mathfrak{g}$ the subgroup of $\mathrm{Aut}\mathfrak{g}$, which contains all the inner automorphisms of $\mathfrak{g}$. For a subgroup $H$ of $\mathrm{Aut}G$, write $G^H:=\{g\in G\mid f(g)=g\textrm{ for all }f\in H\}$. Similarly, for a subgroup $H$ of $\mathrm{Aut}\mathfrak{g}$, write $\mathfrak{g}^H:=\{X\in\mathfrak{g}\mid f(X)=X\textrm{ for all }f\in H\}$. Moreover, if two elements $g_1$ and $g_2$ in the group $G$ are conjugate by an element in the subgroup $H$ of $G$, i.e., $g_2=h^{-1}g_1h$ for some $h\in H$, then write $g_1\sim_Hg_2$; else, write $g_1\nsim_Hg_2$. If $g_1,g_2,\cdots,g_n$ are $n$ elements in a group $G$, then denote by $\langle g_1,g_2,\cdots,g_n\rangle$ the subgroup of $G$ generated by $g_1,g_2,\cdots,g_n$.
\subsection{Outlines}
Symmetric pairs were classified by \'{E}lie Joseph CARTAN and Marcel BERGER in \cite{C1}, \cite{C2}, and \cite{B}. Jingsong HUANG and Jun YU studied semisimple symmetric spaces from a different point of view in \cite{HY}; that is, by determining the Klein four subgroups in the automorphism groups of compact Lie algebras. As an extension of the work of \cite{HY}, Jun YU classified all the elementary abelian 2-groups in the automorphism groups of compact Lie algebras, where groups of rank 2 are just Klein four subgroups. However, the classification of Klein four subgroups in the automorphism groups of non-compact simple Lie algebras is not complete. In \cite{H}, the author classified the Klein four symmetric pairs of holomorphic type for the simple Lie group $\mathrm{E}_{6(-14)}$. In this article, the author studies the Klein four symmetric pairs of holomorphic type for the simple Lie group $\mathrm{E}_{7(-25)}:=\mathrm{Aut}\mathfrak{e}_{7(-25)}$.

The motivation to study the Klein four symmetric pairs of holomorphic type comes from the branching problem raised by Toshiyuki KOBAYASHI in \cite[Problem 5.6]{Ko6}. Concretely, one may want to classify the pairs $(G,G')$ of real reductive Lie groups satisfying the following condition: there exist an infinite dimensional irreducible unitary representation $\pi$ of $G$ and an irreducible unitary representation $\tau$ of $G'$ such that\[0<\dim_\mathbb{C}\mathrm{Hom}_{\mathfrak{g}',K'}(\tau_{K'},\pi_K|_{\mathfrak{g}'})<+\infty\]where $K$ is a maximal compact subgroup of $G$, $\mathfrak{g}$ is the complexified Lie algebra of $G$, $\pi_K|_{\mathfrak{g}'}$ is the restriction of underlying $(\mathfrak{g},K)$-module of $\pi$ to the complexified Lie algebra $\mathfrak{g}'$ of $G'$, $K'$ is a maximal compact subgroup of $G'$ such that $K'\subseteq G'\cap K$, and $\tau_{K'}$ is the underlying $(\mathfrak{g}',K')$-module of $\tau$. For symmetric pairs $(G,G')$, the problem was solved in \cite[Theorem 5.2]{KO}. Unfortunately, it is unknown for general cases, even for Klein four symmetric pairs.

A sufficient condition of it is that $\pi$ is $K'$-admissible by \cite[Proposition 1.6]{Ko4}. For a simple Lie group of Hermitian type $G$ and its reductive subgroup $G'$, if the Lie algebra of $G'$ contains the center of the Lie algebra of a maximal compact subgroup of $G$, then it holds. This is the case when $(G,G')$ is a symmetric pair of holomorphic type classified by Toshiyuki KOBAYASHI and Yoshiki \={O}SHIMA in \cite{KO}, or a symmetric pair of holomorphic type defined in \cite{H} where the author classified the Klein four symmetric pairs of holomorphic type for $\mathrm{E}_{6(-14)}$. In this article, the author classifies Klein four symmetric pairs of holomorphic type for $G=\mathrm{E}_{7(-25)}$, and thus finishes the classification of Klein four symmetric pairs of holomorphic type for exceptional Lie groups of Hermitian type.

The article is organized as follows. At the beginning, the author briefly recalls the definition of Klein four symmetric pairs (of holomorphic type), and the $K'$-admissibility of the restrictions of highest weight representations. After that, the author classifies Klein four symmetric pairs of holomorphic type for $\mathfrak{e}_{7(-25)}$, and applies the results to branching rules.
\section{Preliminary}
\subsection{Klein four symmetric pairs (of holomorphic type)}
The aim of this part is to recall the definition of Klein four symmetric pairs of holomorphic type.
\begin{definition}\label{14}
Let $G$ (respectively, $\mathfrak{g}_0$) be a real simple Lie group (respectively, Lie algebra), and let $\Gamma$ be a Klein four subgroup of $\mathrm{Aut}G$ (respectively, $\mathrm{Aut}\mathfrak{g}_0$). Then $(G,G^\Gamma)$ (respectively, $(\mathfrak{g}_0,\mathfrak{g}_0^\Gamma)$) is called a Klein four symmetric pair. Two Klein four symmetric pairs $(G_1,G'_1)$ (respectively, $(\mathfrak{g}_1,\mathfrak{g}'_1)$) and $(G_2,G'_2)$ (respectively, $(\mathfrak{g}_2,\mathfrak{g}'_2)$) are said to be isomorphic if there exists a Lie group (respectively, Lie algebra) isomorphism $f:G_1\rightarrow G_2$ (respectively, $f:\mathfrak{g}_1\rightarrow\mathfrak{g}_2$) such that $f(G'_1)=G'_2$ (respectively, $f(\mathfrak{g}'_1)=\mathfrak{g}'_2$).
\end{definition}
Suppose that $G$ is a real simple Lie group of Hermitian type; that is, $G/K$ carries a structure of a Hermitian symmetric space where $K$ is a maximal compact subgroup of $G$. Equivalently, the center $Z(\mathfrak{k}_0)$ of Lie algebra $\mathfrak{k}_0$ of $K$ has dimension 1. Take a Cartan involution $\theta$ for $G$, which defines $K$. Let $\tau$ be an involutive automorphism of $G$, which commutes with $\theta$. Use the same letter $\tau$ to denote its differential, and then $\tau$ stabilizes $\mathfrak{k}_0$ and also the center $Z(\mathfrak{k}_0)=\mathbb{R}Z$. Because $\tau^2=1$, there are two possibilities: $\tau Z=Z$ or $\tau Z=-Z$. Recall \cite[Definition 1.4]{Ko5} that the symmetric pair $(G,G^\tau)$ or $(\mathfrak{g}_0,\mathfrak{g}_0^\tau)$ is said to be of holomorphic (respectively, anti-holomorphic) type if $\tau Z=Z$ (respectively, $\tau Z=-Z$), in which case $\tau$ may be said to be of holomorphic (respectively, anti-holomorphic) type for convenience.

According to the classification of the symmetric pairs of holomorphic type and anti-holomorphic type in \cite[Table 3.4.1 \& Table 3,4,2]{Ko5}, there does not exist a symmetric pair which is of both holomorphic type and anti-holomorphic type, so the author remarks that whether a symmetric pair is of holomorphic type or anti-holomorphic type does not depend on the choice of the maximal compact subgroup which is stable under the the action of the involutive automorphism defining the symmetric pair.
\begin{definition}\label{16}
Suppose that $G$ (respectively, $\mathfrak{g}_0$) is a real simple Lie group (respectively, Lie algebra) of Hermitian type. Let $\Gamma=\langle\tau,\sigma\rangle$ be a Klein four subgroup of $\mathrm{Aut}G$ (respectively, $\mathrm{Aut}\mathfrak{g}_0$) generated by $\tau$ and $\sigma$. If both $\tau$ and $\sigma$ are of holomorphic type, then $(G,G^\Gamma)$ (respectively, $(\mathfrak{g}_0,\mathfrak{g}_0^\Gamma)$) is called a Klein four symmetric pair of holomorphic type.
\end{definition}
\subsection{Restrictions of highest weight representations}
For a simple Lie group $G$ of Hermitian type with a maximal compact subgroup $K$, its Lie algebra $\mathfrak{g}_0$ satisfies the condition that a Cartan subalgebra $\mathfrak{h}_0$ of $\mathfrak{k}_0$ becomes a Cartan subalgebra of $\mathfrak{g}_0$. Moreover, there exists a characteristic element $Z\in Z(\mathfrak{k}_0)$ such that $\mathfrak{g}=\mathfrak{l}\oplus\mathfrak{p}_+\oplus\mathfrak{p}_-$ is a decomposition with respect to the eigenspaces of $Z$ on the complexified Lie algebra $\mathfrak{g}$ corresponding to the eigenvalue 0, $\sqrt{-1}$, and $-\sqrt{-1}$ respectively. Similarly, remove the subscript, and then $\mathfrak{h}$ denotes the complexification of $\mathfrak{h}_0$. Choose a positive system $\Phi^+$ for $(\mathfrak{g},\mathfrak{h})$ and its simple system $\Delta$. Denote by $\Phi_\mathfrak{k}^+$ the set of the positive roots of $(\mathfrak{k},\mathfrak{h})$, and write $\Phi_\mathfrak{p}^+:=\Phi^+\setminus\Phi_\mathfrak{k}^+$. Set $\Delta_\mathfrak{k}:=\Phi_\mathfrak{k}^+\cap\Delta$ and $\Delta_\mathfrak{p}:=\Delta\setminus\Delta_\mathfrak{k}$. Obviously, $\Delta_\mathfrak{p}$ contains exactly 1 element because $Z(\mathfrak{k})$ has complex dimension 1.

Suppose that $V$ is a simple $(\mathfrak{g},K)$-module, and then set $V^{\mathfrak{p}_+}=\{v\in V\mid Yv=0\textrm{ for any }Y\in\mathfrak{p}^+\}$. Since $K$ normalizes $\mathfrak{p}_+$, $V^{\mathfrak{p}_+}$ is a $K$-submodule. Further, $V^{\mathfrak{p}_+}$ is either zero or an irreducible finite-dimensional representation of $K$. A $(\mathfrak{g},K)$-module $V$ is called a highest weight module if $V^{\mathfrak{p}_+}\neq\{0\}$.

Recall the definition of unitary highest weight representations. If $\pi$ is a unitary representation of $G$ on a Hilbert space $\mathcal{H}$, and $\mathcal{H}_K$ is the underlying $(\mathfrak{g},K)$-module, then $\pi$ is called a unitary highest weight representation if $\mathcal{H}_K^{\mathfrak{p}_+}\neq\{0\}$; namely, $\mathcal{H}_K$ is a highest weight $(\mathfrak{g},K)$-module.
\begin{proposition}\label{13}
Let $G$ be a real simple Lie group of Hermitian type, and $G'$ a reductive subgroup of $G$. Let $K$ be a maximal compact subgroup of $G$ such that $K':=G'\cap K$ is a maximal compact subgroup of $G'$. Suppose that the Lie algebra $\mathfrak{g}'$ of $G'$ contains the center $Z(\mathfrak{k}_0)$ of the Lie algebra $\mathfrak{k}_0$ of $K$. If $\pi$ is an irreducible unitary highest weight representation of $G$, the restriction $\pi\mid_{G'}$ is $K'$-admissible.
\end{proposition}
\begin{proof}
The conclusion follows from \cite[Theorem 2.9 (1)]{Ko3}, \cite[Example 2.13]{Ko3}, and \cite[Example 3.3]{Ko3}, a brief argument of which is shown in \cite[Proposition 6]{H}.
\end{proof}
\begin{corollary}\label{17}
If $(G,G^{\langle\tau,\sigma\rangle})$ is a Klein four symmetric pair of holomorphic type, then any unitary highest weight representation $\pi$ of $G$ is $K'$-admissible.
\end{corollary}
\begin{proof}
See \cite[Corollary 7]{H}.
\end{proof}
\section{Klein Four Symmetric pairs of holomorphic type for $\mathrm{E}_{7(-25)}$}
\subsection{Elementary abelian 2-subgroups in $\mathrm{Aut}\mathfrak{e}_{6(-78)}$}
The aim of this part is to recall the elementary abelian 2-subgroups in $\mathrm{Aut}\mathfrak{e}_{6(-78)}$, which is helpful to study $\mathrm{E}_{7(-25)}$. Let $\mathfrak{e}_6$ be the complex simple Lie algebra of type $\mathrm{E}_6$. Fix a Cartan subalgebra of $\mathfrak{e}_6$ and a simple root system $\{\alpha_i\mid1\leq i\leq6\}$, the Dynkin diagram of which is given in Figure 1.
\begin{figure}
\centering \scalebox{0.7}{\includegraphics{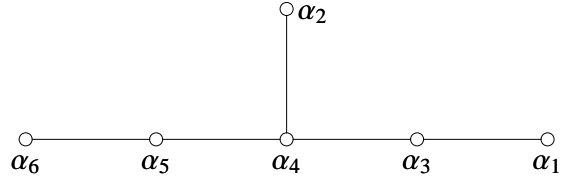}}
\caption{Dynkin diagram of $\mathrm{E}_6$.}
\end{figure}
For each root $\alpha$, denote by $H_\alpha$ its coroot, and denote by $X_\alpha$ the normalized root vector so that $[X_\alpha,X_{-\alpha}]=H_\alpha$. Moreover, one can normalize $X_\alpha$ appropriately such that\[\mathrm{Span}_\mathbb{R}\{X_\alpha-X_{-\alpha},\sqrt{-1}(X_\alpha+X_{-\alpha}),\sqrt{-1}H_\alpha\mid\alpha:\textrm{positive root}\}\cong\mathfrak{e}_{6(-78)}\]is a compact real form of $\mathfrak{g}$ by \cite{Kn}. It is well known that\[\mathrm{Aut}\mathfrak{e}_{6(-78)}/\mathrm{Int}\mathfrak{e}_{6(-78)}\cong\mathrm{Aut}\mathfrak{e}_6/\mathrm{Int}\mathfrak{e}_6\]which is just the automorphism group of the Dynkin diagram.

Follow the constructions of involutive automorphisms of $\mathfrak{e}_{6(-78)}$ in \cite{HY}. Let $\omega$ be the specific involutive automorphism of the Dynkin diagram defined by
\begin{eqnarray*}
\begin{array}{rclcrcl}
\omega(H_{\alpha_1})=H_{\alpha_6},&&\omega(X_{\pm\alpha_1})=X_{\pm\alpha_6},\\
\omega(H_{\alpha_2})=H_{\alpha_2},&&\omega(X_{\pm\alpha_2})=X_{\pm\alpha_2},\\
\omega(H_{\alpha_3})=H_{\alpha_5},&&\omega(X_{\pm\alpha_3})=X_{\pm\alpha_5},\\
\omega(H_{\alpha_4})=H_{\alpha_4},&&\omega(X_{\pm\alpha_4})=X_{\pm\alpha_4}.
\end{array}
\end{eqnarray*}
Let $\tau_1=\mathrm{exp}(\sqrt{-1}\pi H_{\alpha_2})$, $\tau_2=\mathrm{exp}(\sqrt{-1}\pi(H_{\alpha_1}+H_{\alpha_6}))$, $\tau_3=\omega$, $\tau_4=\omega\mathrm{exp}(\sqrt{-1}\pi H_{\alpha_2})$, where $\mathrm{exp}$ represents the exponential map from $\mathfrak{e}_{6(-78)}$ to $\mathrm{Aut}\mathfrak{e}_{6(-78)}$. Then $\tau_1$, $\tau_2$, $\tau_3$, and $\tau_4$ represent all conjugacy classes of involutions in $\mathrm{Aut}\mathfrak{e}_{6(-78)}$, which correspond to real forms $\mathfrak{e}_{6(2)}$, $\mathfrak{e}_{6(-14)}$, $\mathfrak{e}_{6(-26)}$, and $\mathfrak{e}_{6(6)}$.

From \cite{HY}, it is known that $(\mathrm{Int}\mathfrak{e}_{6(-78)})^{\tau_3}\cong\mathrm{F}_{4(-52)}$, the compact Lie group of type $\mathrm{F}_4$, and there exist an involutive automorphism $\eta$ of $\mathrm{F}_{4(-52)}$ such that $\mathfrak{f}_{4(-52)}^{\eta}\cong\mathfrak{sp}(3)\oplus\mathfrak{sp}(1)$, where $\mathfrak{f}_{4(-52)}$ denotes the compact Lie algebra of type $\mathrm{F}_4$. Moreover, $((\mathrm{Int}\mathfrak{e}_{6(-78)})^{\tau_3})^{\eta}\cong\mathrm{Sp}(3)\times\mathrm{Sp}(1)/\langle(-I_3,-1)\rangle$. Let $\mathbf{i}$, $\mathbf{j}$, and $\mathbf{k}$ denote the fundamental quaternion units, and then set $x_0=\tau_3$, $x_1=\eta=(I_3,-1)$, $x_2=(\mathbf{i}I_3,\mathbf{i})$, $x_3=(\mathbf{j}I_3,\mathbf{j})$, $x_4=(\left(\begin{array}{ccc}-1&0&0\\0&-1&0\\0&0&1\end{array}\right),1)$, and $x_5=(\left(\begin{array}{ccc}-1&0&0\\0&1&0\\0&0&-1\end{array}\right),1)$.

For a pair $(r,s)$ of integers with $r\leq2$ and $s\leq3$, define\[F_{r,s}:=\langle x_0,x_1,\cdots,x_s,x_4,x_5,\cdots,x_{r+3}\rangle\]and\[F'_{r,s}:=\langle x_1,x_2,\cdots,x_s,x_4,x_5,\cdots,x_{r+3}\rangle.\]

On the other hand, $(\mathrm{Int}\mathfrak{e}_{6(-78)})^{\tau_4}\cong\mathrm{Sp}(4)/\langle-I_4\rangle$. Set $y_0=\sigma_4$, $y_1=\mathbf{i}I_4$, $y_2=\mathbf{j}I_4$, $y_3=\left(\begin{array}{cc}-I_2&0\\0&I_2\end{array}\right)$, $y_4=\left(\begin{array}{cc}0&I_2\\I_2&0\end{array}\right)$, $y_5=\left(\begin{array}{cccc}1&0&0&0\\0&-1&0&0\\0&0&1&0\\0&0&0&-1\end{array}\right)$, and $y_6=\left(\begin{array}{cccc}0&1&0&0\\1&0&0&0\\0&0&0&1\\0&0&1&0\end{array}\right)$.

For a quadruplet $(u,v,r,s)$ of integers with $u+v\leq1$ and $r+s\leq2$, define\[F_{u,v,r,s}:=\langle y_0,y_1,\cdots,y_{u+2v},y_3,y_4,\cdots,y_{2+2s},y_{3+2s},y_{5+2s},\cdots,y_{1+2r+2s}\rangle\]and\[F'_{u,v,r,s}:=\langle y_1,y_2,\cdots,y_{u+2v},y_3,y_4,\cdots,y_{2+2s},y_{3+2s},y_{5+2s},\cdots,y_{1+2r+2s}\rangle.\]According to \cite[Proposition 6.3 \& Proposition 6.5]{Y}, each elementary abelian 2-subgroup in $\mathrm{Aut}\mathfrak{u}_0$ is conjugate to one of the groups in the one family of $F_{r,s}$, $F'_{r,s}$, $F_{u,v,r,s}$, and $F'_{u,v,r,s}$, and all groups in the families of $F_{r,s}$, $F'_{r,s}$, $F_{u,v,r,s}$, and $F'_{u,v,r,s}$ are pairwisely non-conjugate.
\subsection{Elementary abelian 2-subgroups in $\mathrm{Aut}\mathfrak{e}_{7(-133)}$ induced from $\mathrm{Aut}\mathfrak{e}_{6(-78)}$}
Let $\mathfrak{g}=\mathfrak{e}_7$, the complex simple Lie algebra of type $\mathrm{E}_7$. Fix a Cartan subalgebra of $\mathfrak{e}_7$ and a simple root system $\{\alpha_i\mid1\leq i\leq7\}$, the Dynkin diagram of which is given in Figure 2.
\begin{figure}
\centering \scalebox{0.7}{\includegraphics{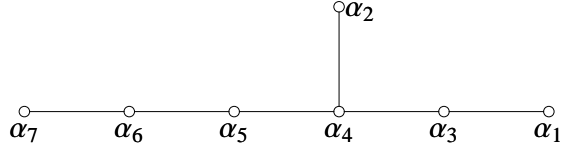}}
\caption{Dynkin diagram of $\mathrm{E}_7$.}
\end{figure}
Let $\mathfrak{u}_0=\mathfrak{e}_{7(-133)}$, the compact simple Lie algebra of type $\mathrm{E}_7$, and $G=\mathrm{Aut}\mathfrak{u}_0$. By \cite{Y}, there are exactly three conjugacy classes of involutive automorphisms in $G$ with representatives $\sigma_1$, $\sigma_2$, and $\sigma_3$ such that \[G^{\sigma_1}\cong(\mathrm{Spin}(12)\times\mathrm{Sp}(1))/\langle(c,1),(-c,-1)\rangle,\] \[G^{\sigma_2}\cong((\mathrm{E}_{6(-78)}\times\mathrm{U}(1))/\langle(c',\mathrm{e}^\frac{2\pi\sqrt{-1}}{3})\rangle)\rtimes\langle\omega\rangle,\] \[G^{\sigma_3}\cong(\mathrm{SU}(8)/\langle\sqrt{-1}I_8\rangle)\rtimes\langle\omega\rangle,\]where $c$ is a certain element in the center of $\mathrm{Spin}(12)$, $c'$ is a nontrivial element in the center of the compact Lie group $\mathrm{E}_{6(-78)}$, and $\omega$ is an involutive automorphism such that $(\mathfrak{e}_{6(-78)}\oplus\sqrt{-1}\mathbb{R})^\omega\cong\mathfrak{f}_4\oplus0$, $\mathfrak{su}(8)^\omega=\mathfrak{sp}(4)$.

The author needs to find all the elementary abelian 2-subgroups in $\mathrm{Aut}\mathfrak{u}_0$ containing $\sigma_2$ because the noncompact dual of $\mathfrak{u}_0$ corresponding to $\sigma_2$ is just $\mathfrak{e}_{7(-25)}$. By \cite{Y}, the elementary abelian 2-subgroups in $\mathrm{Aut}\mathfrak{u}_0$ containing $\sigma_2$ are all induced from $\mathrm{Aut}\mathfrak{e}_{6(-78)}$ as follows.

Denote by $G_{\sigma_2}$ the subgroup of $G^{\sigma_2}$ generated by $\mathrm{E}_{6(-78)}$ and $\omega$. Let $p:G_{\sigma_2}\rightarrow\mathrm{Aut}\mathfrak{e}_{6(-78)}$ be the adjoint homomorphism, and let $i:G_{\sigma_2}\rightarrow G^{\sigma_2}$ be the inclusion map. For any elementary abelian 2-subgroup $\tilde{F}$ of $\mathrm{Aut}\mathfrak{e}_{6(-78)}$, let $F$ be the sylow 2-subgroup of $i(p^{-1}\tilde{F})\times\langle\sigma_2\rangle$. By \cite[Proposition 7.7\&Proposition 7.10]{Y}, $\tilde{F}\mapsto F$ gives one to one correspondence from the conjugacy classes of the elementary abelian 2-subgroups of $\mathrm{Aut}\mathfrak{e}_{6(-78)}$ onto the conjugacy classes of the elementary abelian 2-subgroups of $\mathrm{Aut}\mathfrak{u}_0$ containing $\sigma_2$.

Retain the previous notations, and $\tau_i$ for $1\leq i\leq 4$ lie in $G^{\sigma_2}$. By \cite{Y}, one has\[\tau_1\sim_G\tau_2\sim_G\sigma_1,\]\[\tau_2\sigma_2\sim_G\tau_3\sim_G\tau_3\sigma_2\sim_G\sigma_2,\] \[\tau_1\sigma_2\sim_G\tau_4\sim_G\tau_4\sigma_2\sim_G\sigma_3.\]
\begin{lemma}\label{1}
There are exactly 8 conjugacy classes of the elementary abelian 2-subgroups of rank 3 of $\mathrm{Aut}\mathfrak{u}_0$ containing $\sigma_2$: $\langle x_0,x_1,\sigma_2\rangle$, $\langle x_0,x_4,\sigma_2\rangle$, $\langle x_1,x_2,\sigma_2\rangle$, $\langle x_1,x_4,\sigma_2\rangle$, $\langle x_4,x_5,\sigma_2\rangle$, $\langle y_0,y_1,\sigma_2\rangle$, $\langle y_0,y_3,\sigma_2\rangle$, and $\langle y_3,y_4,\sigma_2\rangle$.
\end{lemma}
\begin{proof}
The conjugacy classes of the Klein four subgroups of $\mathrm{Aut}\mathfrak{e}_{6(-78)}$ are given by \cite[Proposition 6.3\&Proposition 6.5]{Y}, and then the conclusion follows from \cite[Proposition 7.7\&Proposition 7.10]{Y}.
\end{proof}
\subsection{Klein four symmetric pairs of holomorphic type for $\mathfrak{e}_{7(-25)}$}
Let $\mathfrak{g}_0$ be a non-compact simple Lie algebra. For each involutive automorphism $\sigma$ of $\mathfrak{g}_0$, there exists a Cartan involution $\theta$ of $\mathfrak{g}_0$, which commutes with $\sigma$. Denote by $\mathfrak{g}:=\mathfrak{g}_0+\sqrt{-1}\mathfrak{g}_0$ the complexification of $\mathfrak{g}_0$, and write $\mathfrak{u}_0:=\mathfrak{g}_0^\theta+\sqrt{-1}\mathfrak{g}_0^{-\theta}$ for the compact dual of $\mathfrak{g}_0$. Extend $\sigma$ to the unique holomorphic involutive automorphism of $\mathfrak{g}$ and restrict it to $\mathfrak{u}_0$, and then $\sigma$ becomes an involutive automorphism of $\mathfrak{u}_0$. Thus, any involutive automorphism $\sigma$ of $\mathfrak{g}_0$ gives a pair $(\mathfrak{u}_0,\sigma)$, where $\mathfrak{u}_0$ is a compact real form of $\mathfrak{g}$ and $\sigma$ is regarded as an involutive automorphism of $\mathfrak{u}_0$. If $\sigma'$ is conjugate to $\sigma$ in $\mathrm{Aut}\mathfrak{g}_0$, which gives another pair $(\mathfrak{u}'_0,\sigma')$, then there exists an element $g\in\mathrm{Aut}\mathfrak{g}$ such that $g\cdot\mathfrak{u}_0=\mathfrak{u}'_0$ and $g^{-1}\sigma g=\sigma'$; namely, the two pairs $(\mathfrak{u}_0,\sigma)$ and $(\mathfrak{u}'_0,\sigma')$ are conjugate by elements in $\mathrm{Aut}\mathfrak{g}$. Thus, up to conjugations of compact real forms of $\mathfrak{g}$, there exists a map\[\pi:\{\textrm{Conjugacy classes of }\mathrm{Aut}\mathfrak{g}_0\}\longrightarrow\{\textrm{Conjugacy classes of }\mathrm{Aut}\mathfrak{u}_0\}\]which is surjective but not necessarily injective.

If $\Gamma$ is a finite abelian subgroup of $\mathrm{Aut}\mathfrak{g}_0$, there exists a Cartan involution $\theta\in\mathrm{Aut}\mathfrak{g}_0$ centralizing $\Gamma$ by \cite[Proposition 2]{H}. In the same way, $\Gamma$ becomes a finite abelian subgroup of $\mathrm{Aut}\mathfrak{u}_0$. Thus, any finite abelian subgroup of $\mathrm{Aut}\mathfrak{g}_0$ is obtained from a finite abelian subgroup of $\mathrm{Aut}\mathfrak{u}_0$.

On the other hand, fix an involutive automorphism $\theta$ of $\mathfrak{u}_0$, and by holomorphic extension and restriction, $\theta$ is a Cartan involution of a noncompact dual $\mathfrak{g}_0$ of $\mathfrak{u}_0$. Let $\Theta:\mathrm{Aut}\mathfrak{u}_0\rightarrow\mathrm{Aut}\mathfrak{u}_0$ be given by $f\mapsto\theta^{-1}f\theta$, whose differential is $\theta$ on $\mathfrak{u}_0$. Suppose that $\Gamma_1$ and $\Gamma_2$ are two finite abelian subgroups of $\mathrm{Aut}\mathfrak{u}_0$, which are centralized by $\theta$. If $\Gamma_2=g^{-1}\Gamma_1g$ for some $g\in(\mathrm{Aut}\mathfrak{u}_0)^\Theta$, then $\Gamma_1$ and $\Gamma_2$ are conjugate in $\mathrm{Aut}\mathfrak{g}_0$ because $(\mathrm{Aut}\mathfrak{u}_0)^\Theta$ is contained in $\mathrm{Aut}\mathfrak{g}_0$.

Thus, for a compact Lie algebra $\mathfrak{u}_0$, a pair $(\theta,\Gamma)$ with $\theta$ an involutive automorphism of $\mathfrak{u}_0$ and $\Gamma$ a Klein four subgroup of $\mathrm{Aut}\mathfrak{u}_0$ such that $\theta\notin\Gamma$ and $\theta$ centralizes $\Gamma$, gives a Klein four symmetric subalgebra of the noncompact Lie algebra $\mathfrak{g}_0=\mathfrak{u}_0^\theta+\sqrt{-1}\mathfrak{u}_0^{-\theta}$. Moreover, if two such pairs are conjugate $\mathrm{Aut}\mathfrak{u}_0$, then they give a same Klein four symmetric pair up to isomorphism. Here, the pairs $(\theta_1,\Gamma_1)$ and $(\theta_2,\Gamma_2)$ are said to be conjugate in $\mathrm{Aut}\mathfrak{u}_0$ if there exists an element $g\in\mathrm{Aut}\mathfrak{u}_0$ such that $\theta_2=g^{-1}\theta_1g$ and $\Gamma_2=g^{-1}\Gamma_1g$.

Now let $\mathfrak{g}_0=\mathfrak{e}_{7(-25)}$. Retain the notations $\mathfrak{u}_0=\mathfrak{e}_{7(-133)}$ and $G=\mathrm{Aut}\mathfrak{u}_0$ as before. In order to find all the Klein four symmetric pairs of holomorphic type for $\mathfrak{g}_0$, according to the argument above, the author needs to find all the conjugacy classes of the pairs $(\theta,\Gamma)$ with $\theta\in G$ involutive automorphisms and $\Gamma\subseteq G$ Klein four subgroups, which satisfy the following four requirements:
\begin{enumerate}[(i)]
\item $\theta\notin\Gamma$;
\item $\theta\sim_G\sigma_2$;
\item $\theta$ centralizes $\Gamma$;
\item every element $\sigma\in\Gamma$ is identity on the center of $\mathfrak{u}_0^\theta$; namely, each $\sigma\in\Gamma$ gives a symmetric pair of holomorphic type for $\mathfrak{g}_0=\mathfrak{u}_0^\theta+\sqrt{-1}\mathfrak{u}_0^{-\theta}$.
\end{enumerate}
Because of the requirements (i), (ii), and (iii), one knows immediately that each subgroup generated by such $\theta$ and $\Gamma$ must be conjugate to one of the 8 subgroups listed in \lemmaref{1} up to conjugation.

In order to obtain more information, the author makes use of some tools defined in \cite{Y}. For an elementary abelian 2-subgroup $F$ of $G$, define\[H_F:=\{1\}\cup\{x\in F\mid x\sim_G\sigma_1\},\]which is a subgroup by \cite[Lemma 7.3]{Y}. Moreover, define\[m:H_F\times H_F\rightarrow\{\pm1\}\]given by $m(x,y)=-1$, if $\langle x,y\rangle$ is the only Klein four subgroup up to conjugation such that $G^{\langle x,y\rangle}\cong\mathfrak{su}(6)\oplus2(\sqrt{-1}\mathbb{R})$ \cite[Table 4]{HY}; and $m(x,y)=1$ otherwise.

The defect index is defined as\[\mathrm{defe}F:=|\{x\in F\mid x\sim_G\sigma_2\}|-|\{x\in F\mid x\sim_G\sigma_3\}|.\]
\begin{lemma}\label{2}
If the subgroup generated by $\theta$ and $\Gamma$ is conjugate to $\langle x_0,x_1,\sigma_2\rangle$ in $G$, then the pair $(\theta,\Gamma)$ cannot satisfy the requirement (iv).
\end{lemma}
\begin{proof}
Without loss of generality, the author assumes that the group generated by $\theta$ and $\Gamma$ is equal to $\langle x_0,x_1,\sigma_2\rangle$. By the constructions in \cite{Y}, $x_0=\omega=\tau_3\sim_G\sigma_2$, $x_1\sim_{\mathrm{E}_{6(-78)}}\tau_1\sim_G\sigma_1$, $x_0x_1\sim_{\mathrm{E}_{6(-78)}}\tau_3\tau_1\sim_G\sigma_3$, $x_0\sigma_2=\tau_3\sigma_2\sim_G\sigma_2$, $x_1\sigma_2\sim_G\tau_1\sigma_2\sim_G\sigma_3$, and $x_0x_1\sigma_2\sim_G\tau_4\sigma_2\sim_G\sigma_3$. Hence, the elements conjugate to $\sigma_2$ are exactly $\sigma_2$, $x_0$, and $x_0\sigma_2$.

Suppose that $\theta=\sigma_2$. Because $\mathfrak{u}_0^{\sigma_2}=\mathfrak{e}_{6(-78)}\oplus\sqrt{-1}\mathbb{R}$ and $(\mathfrak{e}_{6(-78)}\oplus\sqrt{-1}\mathbb{R})^\omega=\mathfrak{f}_4\subset\mathfrak{e}_{6(-78)}$, $x_0=\omega$ is not identity on the center $\sqrt{-1}\mathbb{R}$ of $\mathfrak{u}_0^{\sigma_2}$. Immediately, neither does $x_0\sigma_2$. But $\Gamma$ contains one of $x_0$ and $x_0\sigma_2$, so $(\theta,\Gamma)$ fails to satisfy the requirement (iv).

Because $x_0$ and $\sigma_2$ commute and are conjugate with each other, $\sigma_2$ is not identity on the center of $\mathfrak{u}_0^{x_0}$, and neither does $x_0\sigma_2$. If $\theta=x_0$, the requirement (iv) fails for the same reason. Similarly, it is not satisfied for $\theta=x_0\sigma_2$.
\end{proof}
\begin{lemma}\label{3}
If the subgroup generated by $\theta$ and $\Gamma$ is conjugate to $\langle x_0,x_4,\sigma_2\rangle$ in $G$, then the pair $(\theta,\Gamma)$ cannot satisfy the requirement (iv).
\end{lemma}
\begin{proof}
Without loss of generality, assumes that the group generated by $\theta$ and $\Gamma$ is equal to $\langle x_0,x_4,\sigma_2\rangle$. Firstly, $x_4\sim_G\sigma_1$ by \cite[Lemma 7.5]{Y}. Secondly, by \cite[Proposition 7.9]{Y}, one obtains that $\mathrm{defe}\langle x_0,x_4,\sigma_2\rangle=6$, and then the elements in $\langle x_0,x_4,\sigma_2\rangle$ are all conjugate to $\sigma_2$ except $1$ and $x_4$.

Let $\theta$ be an element conjugate to $\sigma_2$. Suppose that $\Gamma$ contains $x_4$; namely, $\Gamma=\langle x_4,z\rangle$ for some element $z$ conjugate to $\sigma_2$ other than $\theta$ or $x_4\theta$. Then let $f$ be a group automorphism from $\langle x_0,x_4,\sigma_2\rangle$ to itself, with $f(\theta)=\sigma_2$, $f(x_4)=x_4$, and $f(u)=x_0$. It is obvious that $f(g)\sim_G g$ for all $g\in\langle x_0,x_4,\sigma_2\rangle$. Also, $x_4$ is the unique element conjugate to $\sigma_1$ in $\langle x_0,x_4,\sigma_2\rangle$, so $f$ gives a conjugation between the pairs $(\theta,\Gamma)$ and $(\sigma_2,\langle x_0,x_4\rangle)$ by \cite[Proposition 7.25]{Y}. But $x_0$ is not identity on the center of $\mathfrak{u}_0^{\sigma_2}$, so the pair $(\theta,\Gamma)$ does not satisfy the requirement (iv).

Suppose that $\Gamma$ does not contain $x_4$; namely, $\Gamma=\langle z_1,z_2\rangle$ for two elements $z_1$ and $z_2$ conjugate to $\sigma_2$ such that $z_1z_2\sim_G\sigma_2$ and $\theta$ is not equal to $z_1$, $z_2$, or $z_1z_2$. Then let $f$ be a group automorphism from $\langle x_0,x_4,\sigma_2\rangle$ to itself, with $f(\theta)=\sigma_2$, $f(z_1)=x_4\sigma_2$, and $f(z_2)=x_0$. It is obvious that $f(g)\sim_G g$ for all $g\in\langle x_0,x_4,\sigma_2\rangle$. For the same reason, the pair $(\theta,\Gamma)$ is conjugate to $(\sigma_2,\langle x_0,x_4\rangle)$, and the pair $(\theta,\Gamma)$ does not satisfy the requirement (iv).
\end{proof}
\begin{lemma}\label{4}
If the subgroup generated by $\theta$ and $\Gamma$ is conjugate to $\langle y_0,y_1,\sigma_2\rangle$ in $G$, then the pair $(\theta,\Gamma)$ cannot satisfy the requirement (iv).
\end{lemma}
\begin{proof}
Without loss of generality, assumes that the group generated by $\theta$ and $\Gamma$ is equal to $\langle y_0,y_1,\sigma_2\rangle$. Similarly, one obtains immediately from \cite{Y} that $y_0=\tau_4\sim_G\sigma_3$, $y_3\sim_G\sigma_1$, $y_1\sigma_2\sim_G\sigma_3$, and $y_0\sigma_2=\tau_4\sigma_2\sim_G\sigma_3$. Moreover, by \cite[Proposition 7.9]{Y}, $\mathrm{defe}\langle y_0,y_1,\sigma_2\rangle=-4$. Then $y_0y_1\sim_G\sigma_3$ and $y_0y_1\sigma_2\sim\sigma_3$. Hence, the only element conjugate to $\sigma_2$ is $\sigma_2$ itself in $\langle y_0,y_1,\sigma_2\rangle$. Thus, by \cite[Proposition 7.25]{Y}, a similar argument as the proof for \lemmaref{3} shows that $(\theta,\Gamma)$ is conjugate to either $(\sigma_2,\langle y_0,y_1\rangle)$ for the case $y_1\in\Gamma$, or $(\sigma_2,\langle y_0,y_1\sigma_2\rangle)$ for the case $y_1\notin\Gamma$. In particular, $y_0\in\Gamma$ in both cases. It is known that $y_0=\tau_4=x_1\tau_3$, so $y_0$ is not the identity on the center of $\mathfrak{u}_0^{\sigma_2}$. This completes the proof.
\end{proof}
\begin{lemma}\label{5}
If the subgroup generated by $\theta$ and $\Gamma$ is conjugate to $\langle y_0,y_3,\sigma_2\rangle$ in $G$, then the pair $(\theta,\Gamma)$ cannot satisfy the requirement (iv).
\end{lemma}
\begin{proof}
Without loss of generality, assumes that the group generated by $\theta$ and $\Gamma$ is equal to $\langle y_0,y_3,\sigma_2\rangle$. The proof is parallel to those for the previous lemmas, so the author omits the detailed calculation and just states the sketch. Firstly, $y_0\sim_Gy_0\sigma_2\sim_Gy_0y_3\sim_Gy_0y_3\sigma_2\sim_G\sigma_3$, $y_3\sim_G\sigma_1$, and $y_3\sigma_2\sim_G\sigma_2$. Secondly, any pair $(\theta,\Gamma)$ is conjugate to either $(\sigma_2,\langle y_0,y_3\rangle)$ or $(\sigma_2,\langle y_0,y_3\sigma_2\rangle)$. Thirdly, $y_0$ is not the identity on the center of $\mathfrak{u}_0^{\sigma_2}$. This proof is completed.
\end{proof}
\begin{lemma}\label{6}
Suppose that the subgroup generated by $\theta$ and $\Gamma$ is conjugate to $\langle x_1,x_2,\sigma_2\rangle$ in $G$. Then the pair $(\theta,\Gamma)$ satisfies all of the four requirements if and only if it is conjugate to either $(\sigma_2,\langle x_1,x_2\rangle)$ or $(\sigma_2,\langle x_1,x_2\sigma_2\rangle)$.
\end{lemma}
\begin{proof}
Without loss of generality, assumes that the group generated by $\theta$ and $\Gamma$ is equal to $\langle x_1,x_2,\sigma_2\rangle$. It is known by \cite{Y} that $x_1\sim_Gx_2\sim_Gx_1x_2\sim_G\sigma_1$ and $x_1\sigma_2\sim_Gx_2\sigma_2\sim_Gx_1x_2\sigma_2\sim_G\sigma_3$. Hence, the only element conjugate to $\sigma_2$ is $\sigma_2$ itself in $\langle x_1,x_2,\sigma_2\rangle$. Thus, there are totally four possible pairs: $(\sigma_2,\langle x_1,x_2\rangle)$, $(\sigma_2,\langle x_1\sigma_2,x_2\rangle)$, $(\sigma_2,\langle x_1,x_2\sigma_2\rangle)$, and $(\sigma_2,\langle x_1\sigma_2,x_2\sigma_2\rangle)$.

Let $f$ be a group automorphism from $\langle x_1,x_2,\sigma_2\rangle$ to itself given by $f(\sigma_2)=\sigma_2$, $f(x_1)=x_2$, and $f(x_2)=x_1$. It is obvious that $f(g)\sim_Gg$ for all $g\in\langle x_1,x_2,\sigma_2\rangle$. Moreover, it is obvious that $m(x_1,x_2)=m(x_1,x_1x_2)=m(x_1x_2,x_2)$. Hence by \cite[Proposition 7.25]{Y}, $f$ gives a conjugation between $(\sigma_2,\langle x_1\sigma_2,x_2\rangle)$ and $(\sigma_2,\langle x_1,x_2\sigma_2\rangle)$. In the same way, one shows that $(\sigma_2,\langle x_1\sigma_2,x_2\rangle)$, $(\sigma_2,\langle x_1,x_2\sigma_2\rangle)$, and $(\sigma_2,\langle x_1\sigma_2,x_2\sigma_2\rangle)$ are pairwisely conjugate.

Furthermore, since $\langle x_1,x_2\sigma_2\rangle$ contains an element conjugate to $\sigma_3$ while $\langle x_1,x_2\rangle$ does not, the two subgroups cannot be isomorphic. Therefore, $(\sigma_2,\langle x_1,x_2\rangle)$ or $(\sigma_2,\langle x_1,x_2\sigma_2\rangle)$ are not be conjugate.

Finally, any element in those $\Gamma$ is either in $\mathrm{E}_{6(-78)}$ or of the form $x\sigma_2$ with $x\in\mathrm{E}_{6(-78)}$, and it is obvious that it satisfies the requirement (iv) in either case.
\end{proof}
\begin{lemma}\label{7}
Suppose that the subgroup generated by $\theta$ and $\Gamma$ is conjugate to $\langle x_1,x_4,\sigma_2\rangle$ in $G$. Then the pair $(\theta,\Gamma)$ satisfies all of the four requirements if and only if it is conjugate to one of $(\sigma_2,\langle x_1,x_4\rangle)$, $(\sigma_2,\langle x_1,x_4\sigma_2\rangle)$, and $(\sigma_2,\langle x_1\sigma_2,x_4\rangle)$.
\end{lemma}
\begin{proof}
Without loss of generality, assumes that the group generated by $\theta$ and $\Gamma$ is equal to $\langle x_1,x_4,\sigma_2\rangle$. It is known by \cite{Y} that $x_1\sim_Gx_4\sim_Gx_1x_4\sim_G\sigma_1$, $x_4\sigma_2\sim_G\sigma_2$, and $x_1\sigma_2\sim_Gx_1x_4\sigma_2\sim_G\sigma_3$. Hence, the elements conjugate to $\sigma_2$ are exactly $\sigma_2$ and $x_4\sigma_2$ in $\langle x_1,x_4,\sigma_2\rangle$. Thus, there are totally eight possible pairs: $(\sigma_2,\langle x_1,x_4\rangle)$, $(\sigma_2,\langle x_1\sigma_2,x_4\rangle)$, $(\sigma_2,\langle x_1,x_4\sigma_2\rangle)$, $(\sigma_2,\langle x_1\sigma_2,x_4\sigma_2\rangle)$, $(x_4\sigma_2,\langle x_1,x_4\rangle)$, $(x_4\sigma_2,\langle x_1\sigma_2,x_4\rangle)$, $(x_4\sigma_2,\langle x_1,x_4\sigma_2\rangle)$, and $(x_4\sigma_2,\langle x_1\sigma_2,x_4\sigma_2\rangle)$.

Let $f$ be a group automorphism from $\langle x_1,x_2,\sigma_2\rangle$ to itself given by $f(\sigma_2)=x_4\sigma_2$, $f(x_1)=x_1$, and $f(x_4)=x_4$. It is obvious that $f(g)\sim_Gg$ for all $g\in\langle x_1,x_4,\sigma_2\rangle$. Moreover, it is obvious that $m(x_1,x_4)=m(x_1,x_1x_4)=m(x_1x_4,x_4)$. Hence by \cite[Proposition 7.25]{Y}, $f$ gives a conjugation between $(\sigma_2,\langle x_1,x_4\rangle)$ and $(\sigma_2,\langle x_1,x_4\rangle)$. Similary, one shows that $(\sigma_2,\langle x_1\sigma_2,x_4\rangle)$, $(\sigma_2,\langle x_1,x_4\sigma_2\rangle)$, and $(\sigma_2,\langle x_1\sigma_2,x_4\sigma_2\rangle)$ are conjugate to $(x_4\sigma_2,\langle x_1\sigma_2,x_4\rangle)$, $(x_4\sigma_2,\langle x_1,x_4\sigma_2\rangle)$, and $(x_4\sigma_2,\langle x_1\sigma_2,x_4\sigma_2\rangle)$ respectively. Furthermore, $(\sigma_2,\langle x_1,x_4\sigma_2\rangle)$ and $(\sigma_2,\langle x_1\sigma_2,x_4\sigma_2\rangle)$ are shown to be conjugate in the similar way.

On the other hand, $\langle x_1,x_4\rangle$ contains exactly three elements conjugate to $\sigma_1$, $\langle x_1\sigma_2,x_4\rangle$ contains exactly one element conjugate to $\sigma_1$ and two elements conjugate to $\sigma_3$, and $\langle x_1,x_4\sigma_2\rangle$ contains exactly one element conjugate to $\sigma_1$, one element conjugate to $\sigma_2$, and one element conjugate to $\sigma_3$. Therefore, these $\Gamma$ are pairwisely non-conjugate. It follows that $(\sigma_2,\langle x_1,x_4\rangle)$, $(\sigma_2,\langle x_1,x_4\sigma_2\rangle)$, and $(\sigma_2,\langle x_1\sigma_2,x_4\rangle)$ are pairwisely non-conjugate.

Finally, for the same reason as the proof for \lemmaref{6}, all the pairs satisfy the requirement (iv).
\end{proof}
\begin{lemma}\label{8}
Suppose that the subgroup generated by $\theta$ and $\Gamma$ is conjugate to $\langle x_4,x_5,\sigma_2\rangle$ in $G$. Then the pair $(\theta,\Gamma)$ satisfies all of the four requirements if and only if it is conjugate to either $(\sigma_2,\langle x_4,x_5\rangle)$ or $(\sigma_2,\langle x_4,x_5\sigma_2\rangle)$.
\end{lemma}
\begin{proof}
Without loss of generality, assumes that the group generated by $\theta$ and $\Gamma$ is equal to $\langle x_4,x_5,\sigma_2\rangle$. The proof is parallel to that for \lemmaref{7}, so the author omits the detailed calculation and just states the sketch. Firstly, $x_4\sim_Gx_5\sim_Gx_4x_5\sim_G\sigma_1$ and $x_4\sigma_2\sim_Gx_5\sigma_2\sim_Gx_4x_5\sigma_2\sim_G\sigma_2$. Secondly, any pair $(\theta,\Gamma)$ is conjugate to either $(\sigma_2,\langle x_4,x_5\rangle)$ or $(\sigma_2,\langle x_4,x_5\sigma_2\rangle)$. Thirdly, $(\sigma_2,\langle x_4,x_5\rangle)$ and $(\sigma_2,\langle x_4,x_5\sigma_2\rangle)$ are not conjugate. Finally, they satisfy the four requirements.
\end{proof}
\begin{lemma}\label{9}
Suppose that the subgroup generated by $\theta$ and $\Gamma$ is conjugate to $\langle y_3,y_4,\sigma_2\rangle$ in $G$. Then the pair $(\theta,\Gamma)$ satisfies all of the four requirements if and only if it is conjugate to one of $(\sigma_2,\langle y_3,y_4\rangle)$, $(\sigma_2,\langle y_3,y_4\sigma_2\rangle)$, and $(\sigma_2,\langle y_3\sigma_2,y_4\sigma_2\rangle)$.
\end{lemma}
\begin{proof}
Without loss of generality, assumes that the group generated by $\theta$ and $\Gamma$ is equal to $\langle y_3,y_4,\sigma_2\rangle$. The proof is parallel to that for \lemmaref{7}, so the author omits the detailed calculation and just states the sketch. Firstly, $y_3\sim_Gy_4\sim_Gy_3y_4\sim_G\sigma_1$, $y_3\sigma_2\sim_Gy_4\sigma_2\sim_G\sigma_2$, and $y_3y_4\sigma_2\sim_G\sigma_3$. Secondly, any pair $(\theta,\Gamma)$ is conjugate to one of $(\sigma_2,\langle y_3,y_4\rangle)$, $(\sigma_2,\langle y_3,y_4\sigma_2\rangle)$, and $(\sigma_2,\langle y_3\sigma_2,y_4\sigma_2\rangle)$. Thirdly, the three pairs are pairwisely non-conjugate. Finally, they satisfy the four requirements.
\end{proof}
The following proposition is displayed as a summary of \lemmaref{2} to \lemmaref{9}.
\begin{proposition}\label{10}
The Klein four symmetric pairs of holomorphic type for $\mathfrak{g}_0$ are given by the pairs $(\theta,\Gamma)$ listed as follows, where $\Gamma$ is a Klein four subgroup of $\mathrm{Aut}\mathfrak{u}_0$ and $\theta\in\mathrm{Aut}\mathfrak{u}_0\setminus\Gamma$ is conjugate to $\sigma_2$, such that $\theta$ centralizes $\Gamma$ and every element in $\Gamma$ is identity on the center of $\mathfrak{u}_0^\theta$.
\begin{enumerate}[(1)]
\item $(\sigma_2,\langle x_1,x_2\rangle)$;
\item $(\sigma_2,\langle x_1,x_2\sigma_2\rangle)$;
\item $(\sigma_2,\langle x_1,x_4\rangle)$;
\item $(\sigma_2,\langle x_1,x_4\sigma_2\rangle)$;
\item $(\sigma_2,\langle x_1\sigma_2,x_4\rangle)$;
\item $(\sigma_2,\langle x_4,x_5\rangle)$;
\item $(\sigma_2,\langle x_4,x_5\sigma_2\rangle)$;
\item $(\sigma_2,\langle y_3,y_4\rangle)$;
\item $(\sigma_2,\langle y_3,y_4\sigma_2\rangle)$;
\item $(\sigma_2,\langle y_3\sigma_2,y_4\sigma_2\rangle)$.
\end{enumerate}
\end{proposition}
\begin{proof}
The conclusion follows from \lemmaref{2} to \lemmaref{9}.
\end{proof}
In order to compute the explicit Klein four symmetric pairs of holomorphic type for $\mathfrak{g}_0$, the author needs to calculate $\mathfrak{u}_0^\Gamma$ which is the compact dual of $\mathfrak{g}_0^\Gamma$, and $(\mathfrak{u}_0^\theta)^\Gamma$ which is the maximal compact subalgebra of $\mathfrak{g}_0^\Gamma$.
\begin{lemma}\label{11}
The following isomorphisms of subalgebras in $\mathfrak{u}_0$ hold.
\begin{enumerate}[$\bullet$]
\item $\mathfrak{u}_0^{\langle x_1,x_2\rangle}\cong\mathfrak{u}_0^{\langle y_3,y_4\rangle}\cong\mathfrak{su}(6)\oplus2(\sqrt{-1}\mathbb{R})$;
\item $\mathfrak{u}_0^{\langle x_1,x_4\rangle}\cong\mathfrak{u}_0^{\langle x_4,x_5\rangle}\cong\mathfrak{so}(8)\oplus3\mathfrak{su}(2)$;
\item $\mathfrak{u}_0^{\langle x_4,x_5\sigma_2\rangle}\cong\mathfrak{u}_0^{\langle y_3\sigma_2,y_4\sigma_2\rangle}\cong\mathfrak{so}(10)\oplus2(\sqrt{-1}\mathbb{R})$;
\item $\mathfrak{u}_0^{\langle x_1,x_4\sigma_2\rangle}\cong\mathfrak{u}_0^{\langle y_3,y_4\sigma_2\rangle}\cong\mathfrak{su}(6)\oplus\mathfrak{su}(2)\oplus\sqrt{-1}\mathbb{R}$;
\item $\mathfrak{u}_0^{\langle x_1,x_2\sigma_2\rangle}\cong\mathfrak{u}_0^{\langle x_1\sigma_2,x_4\rangle}\cong2\mathfrak{su}(4)\oplus\sqrt{-1}\mathbb{R}$.
\end{enumerate}
\end{lemma}
\begin{proof}
For each Klein four subgroups $\Gamma$ in $\mathrm{Aut}\mathfrak{u}_0$, the conjugacy classes of the elements are shown in the proof for \lemmaref{2} to \lemmaref{9}. Thus, $\mathfrak{u}_0^\Gamma$ is obtained from the classification of the Klein four symmetric pairs for exceptional Lie algebras \cite[Table 4]{HY} or \cite[Table 3]{Y}. One needs to be careful when $\Gamma$ contains three elements conjugate to $\sigma_1$ because two such subgroups may not be conjugate. However, by \cite[Lemma 7.5]{Y}, $\mathfrak{u}_0^\Gamma$ is still known in this case.
\end{proof}
\begin{lemma}\label{12}
The following isomorphisms of subalgebras in $\mathfrak{u}_0$ hold.
\begin{enumerate}[(i)]
\item $(\mathfrak{u}_0^{\sigma_2})^{\langle x_1,x_2\rangle}=(\mathfrak{u}_0^{\sigma_2})^{\langle x_1,x_2\sigma_2\rangle}\cong2\mathfrak{su}(3)\oplus3(\sqrt{-1}\mathbb{R})$;
\item $(\mathfrak{u}_0^{\sigma_2})^{\langle x_1,x_4\rangle}=(\mathfrak{u}_0^{\sigma_2})^{\langle x_1,x_4\sigma_2\rangle}=(\mathfrak{u}_0^{\sigma_2})^{\langle x_1\sigma_2,x_4\rangle}\cong\mathfrak{su}(4)\oplus2\mathfrak{su}(2)\oplus2(\sqrt{-1}\mathbb{R})$;
\item $(\mathfrak{u}_0^{\sigma_2})^{\langle x_4,x_5\rangle}=(\mathfrak{u}_0^{\sigma_2})^{\langle x_4,x_5\sigma_2\rangle}\cong\mathfrak{so}(8)\oplus3(\sqrt{-1}\mathbb{R})$;
\item $(\mathfrak{u}_0^{\sigma_2})^{\langle y_3,y_4\rangle}=(\mathfrak{u}_0^{\sigma_2})^{\langle y_3,y_4\sigma_2\rangle}=(\mathfrak{u}_0^{\sigma_2})^{\langle y_3\sigma_2,y_4\sigma_2\rangle}\cong\mathfrak{su}(5)\oplus3(\sqrt{-1}\mathbb{R})$.
\end{enumerate}
\end{lemma}
\begin{proof}
See \cite[Lemma 7.4]{Y} and \cite[Lemma 7.5]{Y}.
\end{proof}
\begin{theorem}\label{15}
There are totally 10 Klein for symmetric pairs of holomorphic type $(\mathfrak{g}_0,\mathfrak{g}'_0)$ for $\mathfrak{g}_0=\mathfrak{e}_{7(-25)}$ up to conjugation:
\begin{enumerate}[(1)]
\item $(\mathfrak{e}_{7(-25)},\mathfrak{su}(3,3)\oplus2(\sqrt{-1}\mathbb{R}))$;
\item $(\mathfrak{e}_{7(-25)},2\mathfrak{su}(3,1)\oplus\sqrt{-1}\mathbb{R})$;
\item $(\mathfrak{e}_{7(-25)},\mathfrak{so}^*(8)\oplus\mathfrak{su}(1,1)\oplus2\mathfrak{su}(2))$;
\item $(\mathfrak{e}_{7(-25)},\mathfrak{su}(4,2)\oplus\mathfrak{su}(2)\oplus\sqrt{-1}\mathbb{R})$;
\item $(\mathfrak{e}_{7(-25)},\mathfrak{su}(2,2)\oplus\mathfrak{su}(4)\oplus\sqrt{-1}\mathbb{R})$;
\item $(\mathfrak{e}_{7(-25)},3\mathfrak{su}(1,1)\oplus\mathfrak{so}(8))$;
\item $(\mathfrak{e}_{7(-25)},\mathfrak{so}(8,2)\oplus2(\sqrt{-1}\mathbb{R}))$;
\item $(\mathfrak{e}_{7(-25)},\mathfrak{su}(5,1)\oplus2(\sqrt{-1}\mathbb{R}))$;
\item $(\mathfrak{e}_{7(-25)},\mathfrak{su}(5,1)\oplus\mathfrak{su}(1,1)\oplus\sqrt{-1}\mathbb{R})$;
\item $(\mathfrak{e}_{7(-25)},\mathfrak{so}^*(10)\oplus2(\sqrt{-1}\mathbb{R}))$.
\end{enumerate}
\end{theorem}
\begin{proof}
The conclusion follows from \lemmaref{11} and \lemmaref{12}.
\end{proof}
\subsection{Application to Representation Theory}
As an application of \theoremref{15}, recall \corollaryref{17} that Klein four symmetric pairs of holomorphic type are related to the pairs $(G,G')$ of real reductive groups such that there exist  unitary highest weight representations of $G$, which are $K'$-admissible.

Let $\{e_i\}_{i=1}^n$ be a orthonormal basis for $\mathbb{R}^n$ with respect to the standard inner product, such that $\mathrm{Pin}(n)$ is defined as a subgroup contained in the Clifford algebra $\mathrm{Cl}(\mathbb{R}^n)$. Write $c=e_1e_2\cdots e_n\in\mathrm{Pin}(n)$, and it is known that $c\in\mathrm{Spin}(n)$ if and only if $n$ is even; in this case, $c$ is in the center of $\mathrm{Spin}(n)$ and some noncompact duals $\mathrm{Spin}(m,n-m)$ for $1\leq m\leq n-1$.
\begin{theorem}\label{18}
If $(G,G')$ is one of the pairs listed below, then any irreducible unitary highest weight representation of $G$ is $K'$-admissible, where $K'$ is a maximal compact subgroup of $G'$.
\begin{enumerate}[(i)]
\item $(\mathrm{E}_{7(-25)},\mathrm{SU}(3,3)\times\mathrm{U}(1)^2/\langle(e^{\frac{2\pi\sqrt{-1}}{3}}I_6,e^{-\frac{2\pi\sqrt{-1}}{3}},1),(-I_6,1,1)\rangle)$;
\item $(\mathrm{E}_{7(-25)},\mathrm{Spin}^*(6)^2\times\mathrm{U}(1)/\langle(c,c',1),(1,-1,-1)\rangle)$;
\item $(\mathrm{E}_{7(-25)},\mathrm{Spin}^*(8)\times\mathrm{SU}(1,1)\times\mathrm{SU}(2)^2/\langle(c,-I_2,I_2,I_2),(1,-I_2,-I_2,-I_2),(-1,-I_2,-I_2,I_2)\rangle)$;
\item $(\mathrm{E}_{7(-25)},\mathrm{SU}(4,2)\times\mathrm{SU}(2)\times\mathrm{U}(1)/\langle(e^{\frac{2\pi\sqrt{-1}}{3}}I_6,I_2,e^{-\frac{2\pi\sqrt{-1}}{3}}), (-I_6,-I_2,1)\rangle)$;
\item $(\mathrm{E}_{7(-25)},\mathrm{Spin}(4,2)\times\mathrm{Spin}(6)\times\mathrm{U}(1)^2/ \langle(c,c',1),(1,-1,-1)\rangle)$;
\item $(\mathrm{E}_{7(-25)},\mathrm{SU}(1,1)^3\times\mathrm{Spin}(8)/\langle(-I_2,I_2,I_2,c),(-I_2,-I_2,-I_2,1),(-I_2,-I_2,I_2,-1)\rangle)$;
\item $(\mathrm{E}_{7(-25)},\mathrm{Spin}(8,2)\times\mathrm{U}(1)^2/\langle(c,\sqrt{-1},1)\rangle)$;
\item $(\mathrm{E}_{7(-25)},\mathrm{SU}(5,1)\times\mathrm{U}(1)^2/\langle(e^{\frac{2\pi\sqrt{-1}}{3}}I_6,e^{-\frac{2\pi\sqrt{-1}}{3}},1),(-I_6,1,1)\rangle)$;
\item $(\mathrm{E}_{7(-25)},\mathrm{SU}(5,1)\times\mathrm{SU}(1,1)\times\mathrm{U}(1)/\langle(e^{\frac{2\pi\sqrt{-1}}{3}}I_6,I_2,e^{-\frac{2\pi\sqrt{-1}}{3}}),(-I_6,-I_2,1)\rangle)$;
\item $(\mathrm{E}_{7(-25)},\mathrm{Spin}^*(10)\times\mathrm{U}(1)^2/\langle(c,\sqrt{-1},1)\rangle)$.
\end{enumerate}
\end{theorem}
\begin{proof}
The conclusion follows from \corollaryref{17} and \cite[Table 6]{HY}, and all $G'$ are taken to be the identity components here.
\end{proof}
\begin{corollary}\label{19}
If $(G,G')$ is one of the pairs listed in \theoremref{18}, there exist an infinite dimensional irreducible unitary representation $\pi$ of $G$ and an irreducible unitary representation $\tau$ of $G'$ such that\[0<\dim_\mathbb{C}\mathrm{Hom}_{\mathfrak{g}',K'}(\tau_{K'},\pi_K|_{\mathfrak{g}'})<+\infty\]where $K$ is a maximal compact subgroup of $G$, $\mathfrak{g}$ is the complexified Lie algebra of $G$, $\pi_K|_{\mathfrak{g}'}$ is the restriction of underlying $(\mathfrak{g},K)$-module of $\pi$ to the complexified Lie algebra $\mathfrak{g}'$ of $G'$, $K'$ is a maximal compact subgroup of $G'$ such that $K'\subseteq G'\cap K$, and $\tau_{K'}$ is the underlying $(\mathfrak{g}',K')$-module of $\tau$.
\end{corollary}
\begin{proof}
The conclusion follows from \theoremref{18} and \cite[Proposition 1.6]{Ko4}.
\end{proof}

\end{document}